\documentclass[preprint,12pt]{elsarticle}

\usepackage{graphicx}

\usepackage{amssymb}
\usepackage{amsthm}

\newcommand*{\defeq}{\mathrel{\vcenter{\baselineskip0.5ex \lineskiplimit0pt 
                     \hbox{\scriptsize.}\hbox{\scriptsize.}}}%
                     =} 

\newtheorem{theorem}{Theorem}[section]
\newtheorem{lemma}[theorem]{Lemma}

\theoremstyle{definition}
\newtheorem{definition}[theorem]{Definition}

\theoremstyle{remark}

\journal{Finite Fields and Their Applications}

\begin{document}

\begin{frontmatter}


\title{Polytope Bounds on Multivariate Value Sets} 
\author{Luke Smith}
\ead{smithla@uci.edu}
\address{340 Rowland Hall (Bldg. \# 400), University of California, Irvine. Irvine, CA 92697-3875.}

\title{Polytope Bounds on Multivariate Value Sets}




\begin{abstract}
We improve upon the upper bounds for the cardinality of the value set of a multivariable polynomial map over a finite field using the polytope of the polynomial. This generalizes earlier bounds only dependent on the degree of a polynomial.

\end{abstract}


\end{frontmatter}


\section{Introduction}

For a given polynomial $f(x)$ over a finite field $\mathbb{F}_q$, let $V_f \defeq $ Im$(f) $ denote the value set of $f$. Determining the cardinality and structure of the value set is a problem with a rich history and wide variety of uses in number theory, algebraic geometry, coding theory and cryptography.

\vspace{1pc}

Relevant to this paper are theorems which provide upper bounds on the cardinality of our value set when $f(x)$ is not a permutation polynomial. Let $f(x) \in \mathbb{F}_q[x]$ be a single variable polynomial of degree $d > 0$ with $|V_f| < q$. Using the Chebotarev density theorem over rational function fields, S. D. Cohen proved in \cite{Cone} that there is a finite set of rational numbers $T_d \subset [0,1]$ (depending on degree $d$) such that \begin{eqnarray} \label{Cohen} |V_f| = c_fq + O_d(\sqrt{q}) \end{eqnarray} for some $c_f \in T_d$ depending on Gal($f(x) - t$)/$\mathbb{F}_q(t)$ and Gal($f(x) - t$)/$\overline{\mathbb{F}}_q(t)$. Guralnick and Wan refine this in \cite{Guralnick}, proving that for gcd($d,q$) = 1 and $|V_f| < q$, $|V_f| \leq \frac{47}{63}q + O_d(\sqrt{q})$. In addition, Mullen conjectured the bound \begin{eqnarray} \label{SingleWan} |V_f| \leq q - \frac{q-1}{d} \end{eqnarray} for non-permutation polynomials. This was proven by Wan in \cite{plemma} using $p$-adic liftings, but Turnwald later averted the use of liftings with a clever proof in \cite{Turnwald} using elementary symmetric polynomials. This bound was also proven sharp for any finite field by Cusick and M$\ddot{\textrm{u}}$ller (for $f(x) = (x+1)x^{q-1} \in \mathbb{F}_{q^k}[x]$, $|V_f| = q^k - \frac{q^k-1}{q}$ for all integers $k$, see \cite{sharp}). For more sharp examples, see \cite{WSCsharp}.

\vspace{1pc}

Despite the interest mathematicians have taken in the value set problem, most of the work in this area has been dedicated towards univariate polynomials. However, In the past 25 or so years, the multivariate value set problem has been addressed in a few different forms. It was first addressed by Serre in 1988 \cite{SerreGal} over varieties, in connection with Hilbert's irreducibility theorem and the inverse Galois problem. His theorem, alongside results by Fried \cite{FriedHilb} and by Guralnick and Wan \cite{Guralnick} give us upper bounds on our value set which generalize Cohen's result in (\ref{Cohen}). Though these results bound $|V_f|$ by some fraction of $|\mathbb{F}_q^n|$, it is important to note that the error terms in both results, though well behaved with respect to $q$, are unpredictable in terms of the degree $d$ of the map.

\vspace{1pc}

A recently published paper by Mullen, Wan, and Wang (see \cite{MWWValue}) gives another bound on the value set of polynomial maps, one with no error terms: \begin{eqnarray} \label{MWWBound}\textrm{if }|V_f| < q^n, \textrm{ then }|V_f| \leq q^n - \textrm{min} \left\{ q, \hspace{.25pc} \frac{n(q-1)}{\textrm{deg }f}\right\}. \end{eqnarray} In this paper, we set out to improve upon the above result by generalizing Wan's $p$-adic lifting approach and utilizing the Newton polytope of the multivariate polynomial. We define a quantity $\mu_f$ in section \ref{NewPoly} based on the Newton polytope, one which has the property $\mu_f \geq n / \textrm{deg } f$ (see \cite{Adolph}). The result is as follows: 

\begin{theorem} Let $f(x_1,...,x_n)=(f_1(x_1,...,x_n),...,f_n(x_1,...,x_n))$ be a polynomial vector over the vector space $\mathbb{F}_q^n$. \begin{eqnarray*} \textrm{If } |V_f| < q^n, \textrm{ then }|V_f| \leq q^n - \textnormal{min}\{q,\hspace{.25pc} \mu_f(q-1)\}.\end{eqnarray*} \end{theorem}

\section{The Newton Polytope}
\label{NewPoly}

Let $F$ be an arbitrary field and let $h \in F[x_1,...,x_n]$. If we write $h$ in the form \begin{eqnarray*} h(x_1,...,x_n) = \sum_{j=1}^m a_jX^{V_j}, \hspace{1 pc} a_j \in F^{*} \end{eqnarray*} where \begin{eqnarray} \label{polynota} V_j = (v_{1j},...,v_{nj}) \in \mathbb{Z}_{\geq 0}^n, \hspace{1 pc} X^{V_j} = x_1^{v_{1j}}...x_n^{v_{nj}}, \end{eqnarray} then we have the following definition:

\begin{definition}[Newton polytope] The Newton polytope of polynomial $h \in F[x_1,...,x_n]$, $\Delta(h)$, is the convex closure of the set $\{V_1,...,V_m\} \cup \{(0,...,0)\}$ in $\mathbb{R}^n$.
\end{definition}

Geometric properties of the Newton polytope, such as its dilation by $k \in \mathbb{R}$, its volume or its decomposition into other polytopes via Minkowski Sum, are useful tools in discerning properties of their associated polynomials. For more information, see \cite{WeiCao}, \cite{zeta}, and \cite{variation}.

\vspace{1pc}

The significance of the Newton polytope to the multivariate value set problem comes from the definition of the following quantity:

\begin{definition}[The quantity $\mu_f$] Let $F$ be a field, let $h \in F[x_1,...,x_n]$, and let $\Delta(h)$ be the Newton polytope of $h$. \begin{eqnarray*} \mu_h \defeq \textrm{inf}\{ k \in \mathbb{R}_{>0} \mid k\Delta(h) \cap \mathbb{Z}^n_{>0} \neq \varnothing \}. \end{eqnarray*}
\end{definition}

In other words, $\mu_h$ is the infimum of all positive real numbers $k$ such that the dilation of $\Delta(h)$ by $k$ contains a lattice point with strictly positive coordinates, and we define $\mu_h = \infty$ if such a dilation does not exist. For our purposes, since the vertices of our polytopes have integer coordinates, $\mu_h$ will always be finite and rational so long as we consider $h$ which is not a polynomial in some proper subset of $x_1,..., x_n$. This quantity is used by Adolphson and Sperber \cite{Adolph} to put a lower bound on the $q$-adic valuation ord$_q$ of the number of $\mathbb{F}_q$-rational points on a variety $V$, $N(V)$, over $\mathbb{F}_q$. Namely, let $V = V(f_1,...,f_m)$, where $f_i \in \mathbb{F}_q[x_1,...,x_n]$. If the collection of polynomials $f_1,..., f_m$ is not polynomial in some proper subset of $x_1,..., x_n$, then we have for $f(x_1,..., x_n,y_1,...,y_m) = f_1(x_1,...,x_n)y_1 + \cdots + f_m(x_1,...,x_n)y_m$, \begin{eqnarray*} \textrm{ord}_q (N(V)) \geq  \mu_f - m. \end{eqnarray*}

Note that in the above definitions, the multivariate polynomial $h$ maps the vector space $F^n$ into its base field $F$. However, for the value set problem, we are interested in studying the polynomial vector $f:\mathbb{F}_q^n \longrightarrow \mathbb{F}_q^n$. Fortunately, the definitions we have developed in this section can be extended to polynomial vectors. To properly motivate how this extension arises in our theorem, let us first understand the proof of the univariate result given in (\ref{SingleWan}) and discern how to generalize to the multivariate case in section \ref{From2}.

\section{Single variable value set}
\label{Single}

Before proving our main result, we will provide insight into upper bounds of $|V_f|$ for the case when $f$ is a single variable polynomial. Parts of this proof will generalize to the multivariate case.

\begin{theorem} \label{SingleVar} Let $f(x) \in \mathbb{F}_q[x]$ be a single variable polynomial of degree $d > 0$. If $|V_f| < q$, then \begin{eqnarray*}|V_f| \leq q - \frac{q-1}{d}.\end{eqnarray*}
\end{theorem}

The proof of this theorem relies on the following definition:

\begin{definition}[The quantity $U(f)$] Let $\mathbb{Z}_q$ denote the ring of $p$-adic integers with uniformizer $p$ and residue field $\mathbb{F}_q$. Fix a lifting $\tilde{f}(x) \in \mathbb{Z}_q[x]$ of $f$, taking coefficients from the Teichm$\ddot{\textrm{u}}$ller lifting $L_q \subset $ $\mathbb{Z}_q$ of $\mathbb{F}_q$. Then we define $U(f)$ to be the smallest positive integer $k$ such that the sum \begin{eqnarray*} S_k(f) \defeq \sum_{x \in L_q} \tilde{f}(x)^k \not\equiv 0 \textrm{ (mod } pk). \end{eqnarray*}
\end{definition}

By taking into account the following sum,
\begin{eqnarray} \label{charsum}   \sum_{x \in L_q} x^k = \left\{     \begin{array}{ll}       0, & q-1 \nmid k,\\        q-1, & q-1 \mid k, k \neq 0,\\  q, & k = 0,     \end{array}   \right.\end{eqnarray} 
and remembering that we are only summing over a finite number of terms, we have the following inequality for $f$ not identically zero:

$$ \frac{q-1}{d} \leq U(f) \leq q-1. $$
With the above inequality in mind, Theorem \ref{SingleVar} will follow directly from the following lemma:

\begin{lemma} \label{Ulemma} If $\left|V_f \right| < q$, then \begin{eqnarray*} \left|V_f \right| \leq q - U(f).\end{eqnarray*}
\end{lemma}

The proof of this result is given in the paper by Mullen, Wan, and Wang. See \cite{MWWValue} for the full proof, and see \cite{WSCsharp} for more details regarding this theorem.

\section{From single variable to multivariable}
\label{From2}

Let $f(x_1,...,x_n)=(f_1(x_1,...,x_n),...,f_n(x_1,...,x_n))$ be a polynomial vector, and note $\textrm{deg } f = \textrm{max}_i\{\textrm{deg } f_i\}$. This maps the vector space $\mathbb{F}_q^n$ to itself. Now, take a basis $e_1,...,e_n$ of $\mathbb{F}_{q^n}$ over $\mathbb{F}_q$. Denote $x = x_1e_1 + \cdots + x_ne_n$ and define \begin{eqnarray*}g(x) \defeq f_1(x_1,...,x_n)e_1+ \cdots +f_n(x_1,...,x_n)e_n. \end{eqnarray*} In this way, we can think of the function $g$ as a non-constant univariate polynomial map from the finite field $\mathbb{F}_{q^n}$ to itself. Even better, we have the equality $|V_f| = |g(\mathbb{F}_{q^n})|$. Therefore, using Lemma~\ref{Ulemma}, we know \begin{eqnarray*} \textrm{if } |V_f| < q^n, \textrm{then } |V_f| \leq q^n - U(g),\end{eqnarray*} where $g$ is viewed as a univariate polynomial.

\vspace{1pc}

Unfortunately, as a univariate polynomial, we do not have good control of the univariate degree of $g$ in relation to the multivariate degree of $f$. Even if one were to construct a closed form for $g(x)$ using methods such as Lagrange Interpolation, the degree of $g$ would likely be high enough as to make the resulting upper bound on $|V_f|$ trivial. Because of these issues with the degree of $g$, we cannot use the bounds from the previous section directly, and must rely on another method to bound $U(g)$.

\vspace{1pc}

Previously, we introduced $g(x)$ as a univariate polynomial. However, using a basis $e_1,...,e_n$ of $\mathbb{F}_{q^n}$ over $\mathbb{F}_q$ as before, we can also define a multivariate polynomial \begin{eqnarray*}g(x_1,...,x_n) \defeq f_1(x_1,...,x_n)e_1+ \cdots +f_n(x_1,...,x_n)e_n \end{eqnarray*} mapping the vector space $\mathbb{F}_q^n$ into the field $\mathbb{F}_{q^n}$. In this sense, $g$ as a multivariate polynomial shares some important properties with $f$ as a polynomial vector, such as the fact that deg($g$) = $\textrm{max}_i\{\textrm{deg } f_i\}$. Whereas the paper by Mullen, Wan, and Wang determine a bound for $U(g)$ relying on the multivariate degree of $f$, in this paper we will use the Newton polytope of the multivariate polynomial $g(x_1,...,x_n)$ to improve upon these bounds. With this in mind, we define $\Delta(f) \defeq \Delta(g(x_1,...,x_n))$, $\mu_f \defeq \mu_{g(x_1,...,x_n)}$, and prove the main result of our paper.   

\section{Restatement of Main Theorem and Proof}
\label{Multi}

\begin{theorem} \label{mubound} Let $f(x_1,...,x_n)=(f_1(x_1,...,x_n),...,f_n(x_1,...,x_n))$ be a polynomial vector over the vector space $\mathbb{F}_q^n$. If $|V_f| < q^n$, then \begin{eqnarray*}|V_f| \leq q^n - \textnormal{min}\{q,\hspace{.25pc} \mu_f(q-1)\}.\end{eqnarray*}
\end{theorem}

\begin{proof} First, construct $g$ from our polynomial vector $f$, as we did in Section \ref{From2}. Viewing $g$ as a univariate polynomial $g(x)$, we are allowed to apply Lemma~\ref{Ulemma} to bound $|V_f|$ using $U(g)$. We then consider $g$ as multivariate $g(x_1,...,x_n)$, which allows us to define $\Delta(g)$ and $\mu_g$. Noting that $\Delta(f) = \Delta(g)$ and $\mu_f = \mu_g$ by our definition in Section \ref{From2}, it suffices to prove the following lemma on $U(g)$:
\end{proof}

\begin{lemma} \label{Uglemma} $U(g) \geq \textnormal{min}\{\mu_f(q-1),\hspace{.25pc} q\}.$
\end{lemma}

\begin{proof} Assume the coefficients of $g(x_1,...,x_n)$ are lifted to characteristic zero over $L_{q^n}$, our Teichm$\ddot{\textrm{u}}$ller lifting of $\mathbb{F}_{q^n}$. Remember that $U(g)$ is defined over univariate polynomials to be the smallest positive integer $k$ such that \begin{eqnarray*} S_k(g) \defeq \sum_{x \in L_{q^n}} {g}(x)^k \not\equiv 0 \textrm{ (mod } pk). \end{eqnarray*} However, using $x = x_1e_1 + \cdots + x_ne_n$ as in section \ref{From2}, we can rewrite $S_k(g)$ in terms of multivariate $g(x_1,...,x_n)$. This means $U(g)$ is the smallest positive integer $k$ such that  \begin{eqnarray*} S_k(g) = \sum_{(x_1,...,x_n) \in L_q^n} {g}(x_1,...,x_n)^k \not\equiv 0 \textrm{ (mod } pk). \end{eqnarray*}

Let $k \in \mathbb{Z}_{>0}$ be such that $k < \textnormal{min}\{\mu_f(q-1),\hspace{.25pc} q\}.$ Expand $g(x_1,...,x_n)^k=\sum_{j=1}^m a_jX^{V_j}$ as a polynomial in the $n$ variables $x_1,...,x_n$ (see $(\ref{polynota})$). Since $S_k(g)$ is a finite sum, it can be broken up over the monomials of $g(x_1,...,x_n)^k$. Therefore, it suffices to prove \begin{eqnarray} \label{xvjsum} \sum_{(x_1,...,x_n) \in L_q^n}X^{V_j} \equiv 0 \textrm{ (mod } pk), \hspace{.25pc} 1 \leq j \leq m.\end{eqnarray} If we denote $\ell_j \defeq \#\{v_{ij}, \hspace{.25 pc} 1 \leq i \leq n| v_{ij} \neq 0 \}$, i.e. $\ell_j$ denotes the number of nonzero $v_{ij}$'s with $1 \leq i \leq n$, then we have exactly $n - \ell_j$ zero $v_{ij}$'s, implying that \begin{eqnarray*} \sum_{(x_1,...,x_n) \in L_q^n}X^{V_j} \equiv 0 \textrm{ (mod } q^{n-\ell_j}). \end{eqnarray*} Now let $v_p$ denote the $p$-adic valuation satisfying $v_p(p)=1$. If the inequality \begin{eqnarray*} v_p(q)(n-\ell_j) \geq 1 + v_p(k) \end{eqnarray*} is satisfied, then (\ref{xvjsum}) is true and we are done.

\vspace{1 pc}

Considering $X^{V_j} = x_1^{v_{1j}}...x_n^{v_{nj}}$, the sum on the left side is identically zero if one of the $v_{ij}$ is not divisible by $q-1$ (see (\ref{charsum})). Thus, we shall assume that all $v_{ij}$'s are divisible by $q-1$ (Otherwise (\ref{xvjsum}) is satisfied and we are done without even using our inequality on $k$). Then the total degree of $X^{V_j}$ is \begin{eqnarray*} v_{1j}+ \cdots +v_{nj} \geq (q-1)\ell_j.\end{eqnarray*} Now, the lattice points of $g$ are contained within $\Delta(g)$ by definition, and this implies our lattice points $V_j$ of $g^k$ are contained within $k\Delta(g)$, the dilation of the polytope $\Delta(g)$ by $k$. But since $(q-1) \mid v_{ij}$, we have that $V_j \in (q-1)\mathbb{Z}^n_{\geq 0}$ as well. 

\vspace{1pc}

If we further assume that $V_{j}$ has no zero coordinates, i.e. $\ell_j = n$, this implies \begin{eqnarray*}\left(\frac{k}{q-1}\Delta(g)\right) \cap \mathbb{Z}^n_{>0} \neq \varnothing.\end{eqnarray*} This statement tells us, by the definition of $\mu_f$, that $\frac{k}{q-1} \geq \mu_f$. In other words, \begin{eqnarray*} k \geq \mu_f(q-1). \end{eqnarray*} This contradicts our assumption that $k < \textnormal{min}\{\mu_f(q-1),\hspace{.25pc} q\} \leq \mu_f(q-1).$

\vspace{1pc}

Therefore, when $k < \textnormal{min}\{\mu_f(q-1),\hspace{.25pc} q\},$ we have that $\ell_j < n$, and $n - \ell_j > 0.$ This case, since $k < q,$ gives us $q \nmid k,$ and \begin{eqnarray*} 1 + v_p(k) \leq v_p(q) \leq v_p(q)(n - \ell_j). \end{eqnarray*} This implies that \begin{eqnarray*} S_k(g) \equiv 0 \textrm{ (mod } q^{n-\ell_j}) \equiv 0 \textrm{ (mod } p^{1 + v_p(k)}) \equiv 0 \textrm{ (mod } pk) \end{eqnarray*} and we are done. Lemma~\ref{Uglemma} and the main result of our paper are proved. \end{proof}

It is easy to show that this new bound is an improvement over the previously known result by Mullen, Wan, and Wang described at (\ref{MWWBound}). Adolphson and Sperber gave an elementary proof in \cite{Adolph} that $\mu_f$ $\geq$ $n / \textrm{deg } f$ for all $f$ over any arbitrary vector space $\mathbb{F}_q^n$. For an illustration in two dimensions, please refer to Figure 1 at the end of the text. It can also be shown that this new bound is sharp. Let $f(x_1, x_2) = (x_1, x_1^a x_2)$ with $a$ in $\mathbb{N}$. It is clear that $\Delta(f) \cap \mathbb{Z}^2_{>0} = \{(a,1)\}$, which tells us that $\mu_f = 1$, and we also have $|V_f| = q^2 - (q-1)$. Note that, in general, it is not immediately clear how large of an improvement our new result provides over our previously known bound. Furthermore, an effective method for calculating $\mu_f$ is not directly clear from the definition given. Could there be an efficient way to calculate or estimate $\mu_f$?






\pagebreak

\begin{figure}[h!]
  \centering
\includegraphics[scale=1]{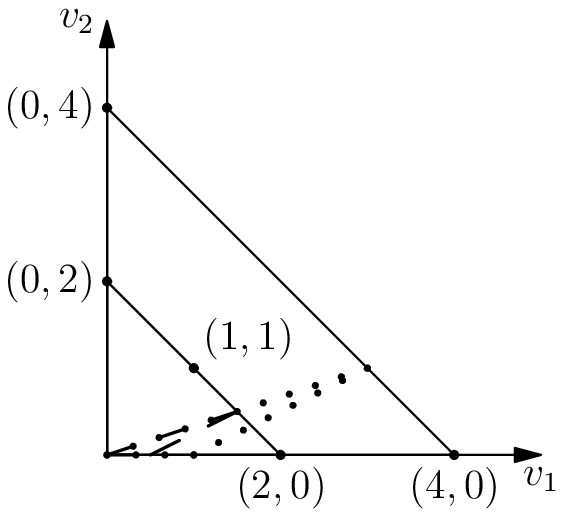}
 \caption{The polytopes of $f(x_1, x_2) = x_1 + x_1^3x_2$ and $h(x_1, x_2) = x_1^4 + x_2^4$, alongside their contractions by $\frac{n}{d} = \frac{2}{4}$. Note that both polynomials are degree 4, $\Delta(f) \cap \mathbb{Z}^2_{>0} = \{(3,1)\}$, and $\left(\frac{2}{4}\Delta(f)\right) \cap \mathbb{Z}^2_{>0} = \varnothing$, but $\left(\frac{2}{4}\Delta(h)\right) \cap \mathbb{Z}^2_{>0} = \{(1,1)\}$. Therefore, $\mu_h = \frac{2}{4} < \mu_f = 1$.}
\end{figure}

\end{document}